\documentclass[11pt]{amsart}
\usepackage{amsmath, amssymb, amscd, bbm, mathrsfs, url, pinlabel, hyperref, verbatim, xargs}
\usepackage[margin=1.25in,centering,letterpaper,dvips]{geometry}
\usepackage{color,dcpic,latexsym,graphicx,epstopdf,comment}
\usepackage[all]{xy}
\usepackage{enumitem}
\usepackage{tikz}
\usetikzlibrary{calc}

\newtheorem {theorem}{Theorem}
\newtheorem {lemma}[theorem]{Lemma}
\newtheorem {proposition}[theorem]{Proposition}
\newtheorem {corollary}[theorem]{Corollary}

\newtheorem {definition}[theorem]{Definition}

\theoremstyle{remark}

\newtheorem {remark}[theorem]{Remark}

\numberwithin{equation}{section}
\numberwithin{theorem}{section}

\newlist{pcases}{enumerate}{1}
\setlist[pcases]{
  label=\bf{Case~\arabic*:}\protect\thiscase.~,
  ref=\arabic*,
  align=left,
  labelsep=0pt,
  leftmargin=0pt,
  labelwidth=0pt,
  parsep=0pt
}
\newcommand{\case}[1][]{%
  \if\relax\detokenize{#1}\relax
    \def\thiscase{}%
  \else
    \def\thiscase{~#1}%
  \fi
  \item
}

\newcommand{\Z}{\mathbb{Z}}
\newcommand{\R}{\mathbb{R}}

\newcommand{\C}{\mathbb{C}}
\newcommand{\F}{\mathbb{F}}

\newcommand{\img}{\operatorname{Im}}

\newcommand{\cB}{\mathcal{B}}
\newcommand{\cC}{\mathcal{C}}

\newcommand{\cT}{\mathcal{T}}

\newcommand{\ssm}{\smallsetminus}

\newcommand{\twosmallmatrix}[4]{{\left(\begin{smallmatrix} #1 & #2 \\ #3 & #4 \end{smallmatrix}\right)}}
\newcommand{\pt}{\mathrm{pt}}

\DeclareMathOperator{\rank}{rank}

\DeclareMathOperator{\sign}{sign}
\DeclareMathOperator{\cross}{cr}
\DeclareMathOperator{\poly}{poly}

\newcommand{\maxnorm}[1]{\left\lVert #1 \right\rVert_\infty}

\newcommand{\NP}{{\sf NP}}
\newcommand{\PP}{{\sf P}}
\newcommand{\coNP}{{\sf co\text{-}NP}}
\newcommand{\unknotrecognition}{{\sf UNKNOT}}
\newcommand{\torusknot}{{\sf TORUS\text{-}KNOT}}
\newcommand{\satelliteknot}{{\sf SATELLITE\text{-}KNOT}}

\tikzset{every picture/.style=thick}
\tikzset{baseline=-\the\dimexpr\fontdimen22\textfont2\relax}

\title{On the complexity of torus knot recognition}
\date{}

\author{John A. Baldwin}
\address{Department of Mathematics \\ Boston College}
\email{john.baldwin@bc.edu}

\author{Steven Sivek}
\address{Department of Mathematics \\ Imperial College London}
\email{s.sivek@imperial.ac.uk}

\begin{document}

\begin{abstract}
We show that the problem of recognizing that a knot diagram represents a specific torus knot, or any torus knot at all, is in the complexity class $\NP \cap \coNP$, assuming the generalized Riemann hypothesis.  We also show that satellite knot detection is in $\NP{}$ under the same assumption, and that cabled knot detection and composite knot detection are unconditionally in $\NP$.  Our algorithms are based on recent work of Kuperberg and of Lackenby on detecting knottedness.
\end{abstract}

\maketitle

\section{Introduction}

According to Thurston \cite{thurston}, nontrivial knots in $S^3$ fall into one of three categories: torus knots, satellite knots, and hyperbolic knots.  Our goal in this paper is to study the computational complexity of the recognition problem for each of these categories, and for torus knots in particular.

Algorithms for unknot recognition have been studied for some time.  This problem was first shown to be decidable by Haken \cite{haken}, using an algorithm based on the theory of normal surfaces.  Hass, Lagarias and Pippenger \cite{hass-lagarias-pippenger} proved that the unknot recognition problem is in \NP{}: this means that given a knot diagram, if the knot is unknotted then there is a certificate which can be used to prove this in polynomial time.  Other proofs come from work of Agol-Hass-Thurston \cite{agol-hass-thurston}, who showed that determining if a knot in any 3-manifold has genus at most $g$ is \NP{}-complete; and from Ivanov \cite{ivanov}, as a corollary of his result that recognizing $S^1 \times D^2$ (among other 3-manifolds) is in \NP{}.

More recently, unknot recognition has also been shown to be in \coNP{}, meaning that given a diagram which does \emph{not} represent the unknot, there is a certificate which can be used to verify its knottedness in polynomial time with respect to the crossing number.  A proof of this was first announced by Agol in 2002 but not published, and later Kuperberg \cite{kuperberg} proved it assuming the generalized Riemann hypothesis (GRH).  Lackenby \cite{lackenby-knottedness} gave the first unconditional proof by showing that the knot genus problem in $S^3$, and more generally the problem of determining the Thurston norm of a homology class in many 3-manifolds, is in \NP{}.  Thus
\[ \unknotrecognition \in \NP \cap \coNP. \]
It is still unknown whether unknot recognition is in \PP{}.

Little else seems to be known about the complexity of other knot recognition problems.  There are algorithms which can decide whether two knots are isotopic, due to work of Haken \cite{haken-homeomorphism}, Hemion \cite{hemion}, and Matveev \cite{matveev}; this also follows from an explicit (though enormous) upper bound of Coward and Lackenby \cite{coward-lackenby} on the number of Reidemeister moves needed to convert one knot diagram to the other.  For the unknot, this bound can be improved to a polynomial in the crossing number \cite{lackenby-polynomial}, giving yet another proof that $\unknotrecognition \in \NP$.  We prove results about the complexity of recognizing any torus knot, as well as detecting whether a knot is a torus knot, a satellite, hyperbolic, cabled, or a connected sum, as follows.

\begin{theorem} \label{thm:main-torus}
The torus knot recognition problem is in $\NP \cap \coNP$, assuming GRH.  Moreover, for any fixed torus knot $T_{r,s}$, the $T_{r,s}$ recognition problem is also in $\NP \cap \coNP$, assuming GRH.
\end{theorem}

\begin{remark}
The claim that $\torusknot \in \NP$ is proved unconditionally in Theorem~\ref{thm:torus-knot-np}, so only the membership in \coNP{} requires GRH.  The same is true for the $T_{r,s}$ recognition problem.
\end{remark}

\begin{theorem} \label{thm:main-satellite}
The satellite knot recognition problem is in \NP{}, assuming GRH.
\end{theorem}

\begin{corollary} \label{cor:main-hyperbolic}
The hyperbolic knot recognition problem is in \coNP{}, assuming GRH.
\end{corollary}

\begin{proof}
We can certify that a knot diagram does not represent a hyperbolic knot by providing a certificate that it is either an unknot, a torus knot, or a satellite knot, since all three recognition problems are in \NP{} (assuming GRH in the case of a satellite knot).
\end{proof}

\begin{theorem} \label{thm:main-cable-composite}
The cabled knot and composite knot recognition problems are in \NP{}.
\end{theorem}

Our proofs of Theorem~\ref{thm:main-torus} and Theorem~\ref{thm:main-satellite} (which appears later as Theorem~\ref{thm:satellite-np}) combine aspects of both Kuperberg's and Lackenby's proofs that $\unknotrecognition \in \coNP$.  In order to recognize a torus knot, a cabled knot, a composite knot, or a satellite knot, we must decompose the knot exterior along essential annuli or tori, certify the incompressibility of the tori in the latter case, and efficiently triangulate what remains; all of this is accomplished in \cite{lackenby-knottedness}.  In order to certify a torus knot, we must then check that what remains is a pair of solid tori, and this can be certified by work of Ivanov \cite{ivanov}; similarly, for Theorem~\ref{thm:main-cable-composite} (a combination of Theorems~\ref{thm:cabled-np} and \ref{thm:composite-np}) we must certify that some of the components are nontrivial knot complements.

The part where Kuperberg's techniques, and hence GRH, play a role in the certification of satellite knots is the verification that the incompressible tori are not boundary parallel.  For this we need to check that some component of their complement has boundary $T^2 \sqcup T^2$ but is not $T^2 \times I$, which we verify using the following result.
\begin{theorem}[Theorem~\ref{thm:solid-torus-knot-rep}] \label{thm:main-solid-torus-rep}
Let $P \subset S^1 \times D^2$ be a knot.  There is a representation
\[ \pi_1((S^1 \times D^2) \ssm P) \to SL_2(\C) \]
with nonabelian image if and only if $P$ is not isotopic to a core of the solid torus.
\end{theorem}
Theorem~\ref{thm:main-solid-torus-rep} is analogous to the theorem of Kronheimer and Mrowka \cite{km-su2} that nontrivial knots in $S^3$ admit nonabelian $SU(2)$ representations, which Kuperberg used to certify knottedness.  (Our proof relies on recent work of Zentner \cite{zentner}, which in turn depends on \cite{km-su2}.)  In both cases, these representations are complex points of algebraic varieties defined over $\Z$, and the use of GRH allows Kuperberg and us to assert that these varieties also have $\F_p$-points where $p$ is a reasonably small prime.  Our certificate that some component is not $T^2 \times I$ is then a nonabelian $SL_2(\F_p)$ representation of its fundamental group.

The proof that $\torusknot \in \coNP$ assuming GRH follows similar lines, and in fact makes use of Theorem~\ref{thm:main-satellite}.  We have already shown how to certify that a non-torus knot $K$ is a satellite knot, but it might be hyperbolic instead.  We use the fact that the peripheral element $\mu^{rs}\lambda$ belongs to the center of the knot group of the torus knot $T_{r,s}$, whereas non-torus knots have trivial center \cite{burde-zieschang}.  We can therefore certify that a knot is not $T_{r,s}$ by finding an $SL_2(\F_p)$ representation $\rho$ of its knot group for which $\rho(\mu^{rs}\lambda)$ is not in the center of the image.  Theorem~\ref{thm:hyperbolic-certificates} asserts that hyperbolic knots admit such certificates: we can find $SL_2(\C)$ representations of this form, e.g.\ any faithful representation, and then the same appeal to GRH provides an $SL_2(\F_p)$ representation as well.

The above argument relies on the specific pair $(r,s)$, but this turns out to not be a problem if we wish to certify that an $n$-crossing diagram $D$ does not represent any torus knot at all, since $D$ can only represent $T_{r,s}$ if $|rs| < 3n$ (see Lemma~\ref{lem:crossing-number-pq}).  We thus need only rule out at most $O(n\log n)$ torus knots $T_{r,s}$ to conclude that $D$ does not represent any torus knot.  In fact, given a polynomial-time method to distinguish any two distinct torus knots (see Lemma~\ref{lem:alexander-signature} and  Proposition~\ref{prop:compute-signature}), this observation shows that the two halves of Theorem~\ref{thm:main-torus} are equivalent: the torus knot recognition problem is in $\NP$ if and only if each $T_{r,s}$ recognition problem is in $\NP$, and likewise for $\coNP$.

Finally, in order to certify that a knot is cabled or composite, after decomposing along an essential annulus we must verify that one or both of the remaining components are nontrivial knot complements rather than solid tori.  In this case we can use Lackenby's work \cite[Theorem~1.5]{lackenby-knottedness} to certify that they have nonzero Thurston norm, and this shows that they are not solid tori as needed.

\subsection*{Organization}
In Section~\ref{sec:alex-sig} we show that the Alexander polynomial and signature uniquely determine each torus knot among the set of all torus knots, and that they can be computed in polynomial time.  After a review of some facts about normal surfaces in Section~\ref{sec:normal-surfaces}, Section~\ref{sec:torus-knot-np} proves half of Theorem~\ref{thm:main-torus}, namely that the torus knot recognition problem is in \NP{}.  Section~\ref{sec:hyperbolic-certificates} shows how to certify that hyperbolic knots are not torus knots; we explain in detail the role played by representation varieties and GRH, following \cite{kuperberg}.  In Section~\ref{sec:satellite-np} we combine techniques from Sections~\ref{sec:torus-knot-np} and \ref{sec:hyperbolic-certificates} to prove Theorem~\ref{thm:main-satellite} and complete the proof of Theorem~\ref{thm:main-torus}.  Finally, in Section~\ref{sec:cable-composite} we prove Theorem~\ref{thm:main-cable-composite}, showing that the cable and composite knot recognition problems are also in $\NP{}$.

\subsection*{Acknowledgments}
We are grateful to the referee for providing useful feedback which improved the quality of the paper.  JAB was supported by NSF Grant DMS-1406383 and NSF CAREER Grant DMS-1454865.  SS would like to thank the Max Planck Institute for Mathematics for its hospitality during some of the period in which this paper was completed.

\section{Alexander polynomials and signatures of torus knots} \label{sec:alex-sig}

In this section we show that two torus knots are isotopic if and only if they have the same Alexander polynomial and signature, and that given a diagram for $K$ there is a polynomial time algorithm to either find the unique torus knot $T$ such that $\Delta_K(t) = \Delta_T(t)$ and $\sigma(K) = \sigma(T)$ or determine that $T$ does not exist.  This will imply that the two claims of Theorem~\ref{thm:main-torus} are equivalent, namely that $\torusknot$ is in $\NP \cap \coNP$ if and only if the $T_{r,s}$ recognition problem is for each pair $r,s$, because of the following bound.

\begin{lemma} \label{lem:crossing-number-pq}
If $\Delta_K(t) = \Delta_{T_{r,s}}(t)$ for some integers $r > s \geq 2$, then $rs < 3\cross(K)$ where $\cross$ denotes crossing number.
\end{lemma}

\begin{proof}
To see this, we use the well-known fact that $2\deg(\Delta_K(t)) \leq \cross(K)$ (see Exercise~4 of \cite[Chapter~VIII]{crowell-fox}), which in this case gives $(r-1)(s-1) \leq \cross(K)$.  Combining $rs \leq \cross(K)+r+s-1$ with $r+s \leq \frac{1}{2}rs+2$ (which is equivalent to $(r-2)(s-2)\geq 0$) gives $rs \leq \cross(K)+\frac{1}{2}rs+1$, or $rs \leq 2\cross(K)+2$.  But $K$ is nontrivial since $\Delta_K(t) \neq 1$, so its crossing number is at least 3 and hence the right hand side is less than $3\cross(K)$.
\end{proof}

As explained in the introduction, Lemma~\ref{lem:crossing-number-pq} reduces the torus knot recognition problem to the collection of $T_{r,s}$ recognition problems, since we can just check that a given $n$-crossing diagram does or does not represent $T_{r,s}$ for each pair of coprime integers $(r,s)$ with $|r| > s \geq 2$ and $|rs| < 3n$.  There are only $O(n\log n)$ such pairs, since for each value of $r$ there are at most $\frac{3n}{|r|}$ possible values of $s$, so if each one can be checked individually in polynomial time then so can all of them.


\begin{lemma} \label{lem:alexander-signature}
If $K$ and $K'$ are torus knots with the same Alexander polynomial and signature, then $K$ is isotopic to $K'$.
\end{lemma}

\begin{proof}
We first claim that $K$ is determined up to chirality by its Alexander polynomial.  Supposing that $K = T_{r,s}$ for some relatively prime $r,s \geq 2$, it is well-known (see e.g.\ \cite{lickorish}) that $T_{r,s}$ has genus $g = \frac{(r-1)(s-1)}{2}$ and Alexander polynomial
\[ \Delta_K(t) = \frac{1}{t^g} \frac{(t^{rs}-1)(t-1)}{(t^r-1)(t^s-1)}, \]
whose leading term is $t^g$.  Thus $\Delta_K(t)$ determines both $rs$, since the least value of $\theta > 0$ for which $\Delta_K(e^{i\theta})=0$ is $\theta = \frac{2\pi}{rs}$, and $(r-1)(s-1)$, as twice the degree of its leading term.  From these we can deduce the value of
\[ r+s = rs - (r-1)(s-1) + 1 \]
as well, and then $r$ and $s$ are uniquely determined as the roots of $x^2 - (r+s)x + rs$.

The Alexander polynomial does not distinguish between a torus knot $T_{r,s}$ and its mirror $T_{-r,s}$, but the signature does; we adopt the sign convention for which the right-handed trefoil has signature $-2$.  Rudolph \cite{rudolph} showed that braid-positive knots have negative signature, so if $K$ is a torus knot with $\Delta_K(t) = \Delta_{T_{r,s}}(t)$ for some $r,s \geq 2$, then it follows that $K = T_{r,s}$ if $\sigma(K) < 0$ and $K = T_{-r,s}$ if $\sigma(K) > 0$.
\end{proof}

Here, and throughout the rest of this paper, we write ``$y = \poly(x_1,\dots,x_n)$'' to indicate that the value of $y$ is bounded above by some fixed polynomial in $x_1$ through $x_n$.

Since we are using the Alexander polynomial and signature to distinguish torus knots, we need to know that they can be computed from a diagram $D$ with $n$ crossings in $\poly(n)$ time.  We first recall some facts from computational algebra.  Given a polynomial $f \in \Z[t]$, we will write $\maxnorm{f}$ to denote the largest absolute value of any of its coefficients.

\begin{lemma} \label{lem:det-over-zt}
Let $A$ be a $k\times k$ matrix with entries $a_{ij} \in \Z[t]$, each of which has degree at most $d \geq 1$ and satisfies $\maxnorm{a_{ij}} \leq C$ for some constant $C$.  Then $\det(A)$ can be computed in $\poly(dk,\log C)$ time.
\end{lemma}

\begin{proof}
For each $i=0,1,\dots,dk$, we let $A_i = A|_{t=i}$ denote the matrix obtained from $A$ by substituting $t=i$; then each entry of $A_i$ has absolute value at most
\[ (d+1)Ci^d \leq (d+1)C(dk)^d. \]
The determinant of an $k\times k$ integer matrix with all entries between $-C'$ and $C'$ can be computed in $\poly(k,\log C')$ time, see e.g.\ \cite[Theorem~5.12]{gathen-gerhard}.  Thus $\det(A_i)$ can be computed in $\poly(k, \log((d+1)C(dk)^d)) = \poly(dk, \log C)$ time.  Hadamard's inequality \cite[Theorem~16.6]{gathen-gerhard} says that
\[ \left|\det(A_i)\right| \leq k^{k/2}\big((d+1)Ci^d\big)^k, \]
so that $\log \left|\det(A_i)\right| = O(dk\log dk)$ for $0 \leq i \leq dk$.

Since $\det(A)$ is a polynomial $f(t)$ of degree at most $dk$, it is determined by the values $f(i) = \det(A_i)$ for $i=0,\dots,dk$.  We can compute each of these in a total of $\poly(dk,\log(C))$ time, and then we reconstruct $\det(A)=f(t)$ by Lagrange interpolation, using a total of $O((dk)^3)$ arithmetic operations on integers $\det(A_i)$ with $O(dk\log dk)$ digits each.
\end{proof}

Lemma~\ref{lem:det-over-zt} implies that the characteristic polynomial of a $k\times k$ integer matrix $A$ with entries between $-C$ and $C$ can be computed as $\det(tI-A)$ in $\poly(k,\log C)$ time.

\begin{lemma} \label{lem:gcd-over-zt}
Let $f_1,\dots,f_k \in \Z[t]$ be polynomials of degree at most $d\geq 1$ and with $\maxnorm{f_i} \leq C$ for some constant $C$.  Then $\gcd(f_1,\dots,f_k)$ can be computed in $\poly(k,d,\log C)$ time.
\end{lemma}

\begin{proof}
In the case $k=2$, we can compute $\gcd(f_1,f_2)$ in $\poly(d,\log C)$ time according to \cite[Theorem~6.62]{gathen-gerhard}.  We would like to apply this algorithm repeatedly in the general case, by computing $g_i := \gcd(f_1,\dots,f_i)$ as the greatest common divisor of $g_{i-1} = \gcd(f_1,\dots,f_{i-1})$ and $f_i$ for each $i=2,\dots,k$, but we need to ensure that the intermediate norms $\maxnorm{g_{i-1}}$ stay reasonably small.

To that end, suppose we have polynomials $p,q,r \in \Z[t]$ of degrees $a \geq 1,b,c$, where $qr$ divides $p$ and hence $b+c \leq a$.  Then Mignotte's bound \cite[Corollary~6.33]{gathen-gerhard} says that
\[ \maxnorm{q}\maxnorm{r} \leq (a+1)^{1/2}2^{b+c} \maxnorm{p}. \]
In the above case, since $g_{i-1}$ divides $f_1$, we have $\deg(g_{i-1}) \leq d$ and $\maxnorm{g_{i-1}} \leq (d+1)^{1/2}2^d C$, and so we can compute $g_i = \gcd(g_{i-1},f_i)$ in 
\[ \poly\left(d, \log\big((d+1)^{1/2}2^dC\big)\right) = \poly(d, \log C) \]
time.  This is independent of $i$, so after $k$ steps we can determine $g_k$ in the claimed time.
\end{proof}

\begin{proposition} \label{prop:compute-signature}
Given a diagram $D$ with $n$ crossings representing the knot $K$, there is an algorithm which computes $\Delta_K(t)$ and $\sigma(K)$ in time $\poly(n)$.
\end{proposition}

\begin{proof}
In order to compute $\Delta_K(t)$, we take the Wirtinger presentation associated to $D$ and compute the associated $n \times n$ matrix of Fox derivatives in $O(n^2)$ time.  We abelianize this to get a matrix with entries in $\Z[t^{\pm 1}]$, and then $\Delta_K(t)$ is the greatest common divisor of the determinants of the $(n-1)\times(n-1)$ minors of this matrix.  We multiply the minors by a power of $t$ so their entries lie in $\Z[t]$ and compute their determinants using Lemma~\ref{lem:det-over-zt}.  The greatest common divisor can then be computed in time $\poly(n)$ by Lemma~\ref{lem:gcd-over-zt}, since all of the coefficients and degrees of polynomials which occur along the way have length at most $\poly(n)$.  We then normalize $\Delta_K(t)$ so that it is a symmetric Laurent polynomial and $\Delta_K(1) = 1$.

For $\sigma(K)$, we use Gordon and Litherland's signature formula \cite{gordon-litherland-signature}, which says that $\sigma(K) = \sign(G) - \mu$, where $G$ is a Goeritz matrix for $D$ and $\mu$ is a certain correction term.  To construct these, we orient $D$ and choose a checkerboard coloring of $D$, label the white regions $X_0,\dots,X_k$, and assign each crossing $c$ a sign $\eta(c) = \pm 1$ and a type (I or II) as shown below.
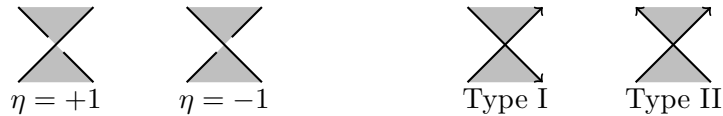
\[
\begin{tikzpicture}
    \fill[fill=lightgray] (-0.5,0.5)--(0.5,0.5)--(-0.5,-0.5)--(0.5,-0.5);
    \draw (-0.1,0.1) -- (-0.5,0.5);
    \draw (-0.5,-0.5) -- (0.5,0.5);
    \draw (0.5,-0.5) -- (0.1,-0.1);
    \node at (0,-0.75) {$\eta = +1$};
\end{tikzpicture}
\qquad
\begin{tikzpicture}
    \fill[fill=lightgray] (-0.5,0.5)--(0.5,0.5)--(-0.5,-0.5)--(0.5,-0.5);
    \draw (0.1,0.1) -- (0.5,0.5);
    \draw (0.5,-0.5) -- (-0.5,0.5);
    \draw (-0.5,-0.5) -- (-0.1,-0.1);
    \node at (0,-0.75) {$\eta = -1$};
\end{tikzpicture}
\qquad\qquad\qquad
\begin{tikzpicture}
    \fill[fill=lightgray] (-0.5,0.5)--(0.5,0.5)--(-0.5,-0.5)--(0.5,-0.5);
    \draw[->] (-0.5,0.5) -- (0.5,-0.5);
    \draw[->] (-0.5,-0.5) -- (0.5,0.5);
    \node at (0,-0.75) {Type I};
\end{tikzpicture}
\qquad
\begin{tikzpicture}
    \fill[fill=lightgray] (-0.5,0.5)--(0.5,0.5)--(-0.5,-0.5)--(0.5,-0.5);
    \draw[->] (0.5,-0.5) -- (-0.5,0.5);
    \draw[->] (-0.5,-0.5) -- (0.5,0.5);
    \node at (0,-0.75) {Type II};
\end{tikzpicture}
\]
We then build a $(k+1) \times (k+1)$ matrix $G'$ whose off-diagonal entries $g_{ij}$, $0 \leq i,j \leq k$, are defined as $-\sum \eta(c)$ over all crossings $c$ incident to both $X_i$ and $X_j$, and for which $g_{ii} = -\sum_{j\neq i} g_{ij}$.  The Goeritz matrix $G$ is the $k\times k$ submatrix $\big(g_{ij}\big)_{1\leq i,j \leq k}$, and then we let $\mu = \sum \eta(c)$ over all type II crossings $c$.

From the construction we see that $k \leq n$: forgetting the crossing information turns $D$ into a planar 4-valent graph with $n$ vertices, and hence $2n$ edges and $2n - n + 2 = n+2$ regions in its complement, so at most $n+1$ are white.  The entries of $G$ have magnitude at most $n$, since each $|g_{ij}|$ or $|g_{ii}|$ is bounded by the number of crossings incident to $X_i$, and clearly $|\mu| \leq n$ as well.  Both $G$ and $\mu$ are easily computable in polynomial time, so it remains to be seen that the signature $\sign(G)$ is as well.

Since $G$ is symmetric and $\det(G) = \pm\det(K) \neq 0$ (see e.g.\ \cite[Corollary~9.5]{lickorish}), the eigenvalues of $G$ are all real and nonzero.  The entries of $G$ are all bounded by $n$ in magnitude, so we can compute the characteristic polynomial
\[ f(\lambda) = \det(\lambda I - G) = \lambda^k + a_{1} \lambda^{k-1} + \dots + a_{k-1}\lambda + a_k \]
in time $\poly(k,\log n) = \poly(n)$.  We let $k_+$ denote the number of sign changes in the sequence $a_0=1, a_1, \dots, a_k$, meaning the number of indices $i$ such that $a_ia_j < 0$ for some $j>i$ and all intermediate terms $a_{i+1},\dots,a_{j-1}$ are zero; likewise we let $k_-$ denote the number of sign changes in $(-1)^ka_0, (-1)^{k-1}a_1, (-1)^{k-2}a_2, \dots, a_k$.  Since all $k$ roots of $f(\lambda)$ are real and nonzero, and the total number of sign changes in both sequences is easily seen to be at most $k$, we have
\[  k = \#\{\mathrm{positive\ roots}\} + \#\{\mathrm{negative\ roots}\} \leq k_+ + k_- \leq k \]
where the first inequality is Descartes' rule of signs.  These inequalities must then be equalities, so  $f(\lambda)$ has exactly $k_\pm$ roots of either sign, and thus $\sign(G) = k_+ - k_-$.
\end{proof}

\begin{proposition} \label{prop:which-torus-knot}
Given a diagram $D$ for a knot $K$, there is an algorithm which returns the unique pair $(r,s)$ of coprime integers with $|r|>s \geq 2$ such that $\Delta_K(t) = \Delta_{T_{r,s}}(t)$ and $\sigma(K) = \sigma(T_{r,s})$ if such a pair exists, and the empty set otherwise.  If $D$ has $n$ crossings, then this algorithm runs in time $\poly(n)$.
\end{proposition}

\begin{proof}
We use $D$ to compute the Alexander polynomial $\Delta_K(t)$ and signature $\sigma(K)$ in $\poly(n)$ time, by Proposition~\ref{prop:compute-signature}.  If $\Delta_K(t) = \Delta_{T_{a,b}}(t)$ where $a>b\geq 2$, then $\Delta_K(t)$ has maximal degree $d = \frac{(a-1)(b-1)}{2}$.  Lemma~\ref{lem:crossing-number-pq} says that $ab < 3n$, so $b < \sqrt{3n}$.  Thus for each $b$ in the range $2 \leq b < \sqrt{3n}$, we set $a = \frac{2d}{b-1}+1$, and if this is an integer greater than $b$ then we compute
\[ (t^a-1)(t^b-1)\Delta_K(t) - (t^{ab}-1)(t-1). \]
If this is zero for some such pair $(a,b)$, then $\Delta_K(t) = \Delta_{T_{a,b}}(t)$; and if not, then $\Delta_K(t)$ is not the Alexander polynomial of a torus knot and so we return $\emptyset$.  This requires $O(\sqrt{n})$ operations on polynomials of degree at most $ab < 3n$, and hence can also be done in $\poly(n)$ time.

Next, assuming that $\Delta_K(t) = \Delta_{T_{a,b}}(t)$ as above, we compute $\sigma(T_{a,b})$ using Litherland's formula \cite{litherland}.  This requires only $O(ab)$ operations on integers of size at most $O(ab)$; we note that $O(ab)$ is in fact $O(n)$ if we reach this step, since $ab < 3n$.  If $\sigma(K) = \pm \sigma(T_{a,b})$ then we return $(r,s) = (\pm a,b)$, and otherwise we return $\emptyset$.
\end{proof}

\section{Normal surfaces and essential annuli} \label{sec:normal-surfaces}

Our recognition algorithms will rely heavily on being able to identify annuli and tori in knot complements which are essential, meaning that they are incompressible and boundary-incompressible.  In this section we review key results which will allow us to do so.  We first recall what it means to be a normal surface in a triangulated 3-manifold.

Let $\cT$ be a triangulation of a compact 3-manifold $Y$.  A properly embedded surface $S \subset Y$ is \emph{normal} if it intersects each tetrahedron of $\cT$ in a disjoint union of finitely many \emph{elementary disks}, which are properly embedded triangles or quadrilaterals.  There are seven types of elementary disks, one triangle per vertex and one quadrilateral per pair of opposite edges:
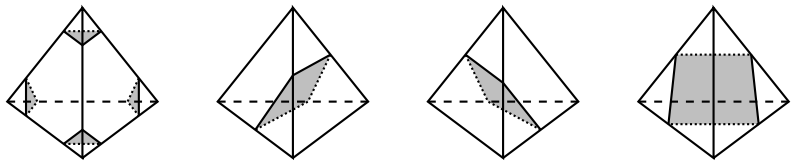
\[
\def\drawfronttriangle(#1)(#2)(#3){%
  \fill[fill=lightgray] (#1) -- (#2) -- (#3) -- cycle;
  \draw (#1) -- (#2) -- (#3);
  \draw[densely dotted] (#1) -- (#3);
}
\def\drawbacktriangle(#1)(#2)(#3){%
  \fill[fill=lightgray] (#1) -- (#2) -- (#3) -- cycle;
  \draw (#1) -- (#2);
  \draw[densely dotted] (#2) -- (#3) -- (#1);
}
\def\drawrectangle(#1)(#2)(#3)(#4){%
  \fill[fill=lightgray] (#1) -- (#2) -- (#3) -- (#4) -- cycle;
  \draw (#1) -- (#2) -- (#3);
  \draw[densely dotted] (#3) -- (#4) -- (#1);
}
\def\drawbackrectangle(#1)(#2)(#3)(#4){%
  \fill[fill=lightgray] (#1) -- (#2) -- (#3) -- (#4) -- cycle;
  \draw (#1) -- (#2);
  \draw[densely dotted] (#2) -- (#3);
  \draw (#3) -- (#4);
  \draw[densely dotted] (#4) -- (#1);
}
\def\drawtetrahedron(#1)(#2)(#3)(#4){%
  \draw (#3) -- (#1) -- (#2) -- (#3) -- (#4) -- (#2);
  \draw[dashed] (#1) -- (#4);
}
\begin{tikzpicture}
    \drawfronttriangle(-0.25,1.6875)(0,1.5)(0.25,1.6875);
    \drawfronttriangle(-0.25,0.1875)(0,0.375)(0.25,0.1875);
    \drawbacktriangle(-0.75,0.5625)(-0.75,1.0625)(-0.6,0.75);
    \drawbacktriangle(0.75,0.5625)(0.75,1.0625)(0.6,0.75);
    \drawtetrahedron(-1,0.75)(0,0)(0,2)(1,0.75);
\end{tikzpicture}
\qquad
\begin{tikzpicture}
   \drawrectangle(-0.5,0.375)(0,1.1)(0.5,1.375)(0.2,0.75);
   \drawtetrahedron(-1,0.75)(0,0)(0,2)(1,0.75);
\end{tikzpicture}
\qquad
\begin{tikzpicture}
   \drawrectangle(-0.5,1.375)(0,1)(0.5,0.375)(-0.2,0.75);
   \drawtetrahedron(-1,0.75)(0,0)(0,2)(1,0.75);
\end{tikzpicture}
\qquad
\begin{tikzpicture}
   \drawbackrectangle(-0.6,0.45)(-0.5,1.375)(0.5,1.375)(0.6,0.45);
   \drawtetrahedron(-1,0.75)(0,0)(0,2)(1,0.75);
\end{tikzpicture}
\]
If there are $t$ tetrahedra, then we can describe a normal surface $S$ by a vector $v(S) \in \Z^{7t}$ with nonnegative coordinates, counting the number of parallel copies of each type of elementary disk in each tetrahedron.  In fact, a vector $v \in \Z^{7t}$ describes an embedded surface if and only if its entries are nonnegative; it satisfies the \emph{matching equations}, linear equations asserting that when we glue tetrahedra together along a pair of faces, the edges of the elementary disks in each tetrahedron form the same pattern of line segments on either face; and at most one type of quadrilateral appears in each tetrahedron.

The set of all vectors in $\R^{7t}$ satisfying the nonnegativity and matching conditions is a polyhedral cone, called the \emph{Haken normal cone} of the triangulation.  The \emph{carrier} of a normal surface $S$ is the minimal face $\cC(S)$ of the Haken normal cone containing $v(S)$, and $S$ is said to be a \emph{vertex surface} if it is connected and two-sided and its carrier is a 1-dimensional face.  We say a vertex surface $S$ is \emph{minimal} if $v(S)$ is not a nontrivial integral multiple of another integral point on this face.  In this case there is a bound on its complexity, due to Hass, Lagarias, and Pippenger \cite{hass-lagarias-pippenger}.

\begin{theorem}[{\cite[Lemma~6.1]{hass-lagarias-pippenger}}] \label{thm:hlp-bound}
If $S$ is a minimal vertex surface with respect to a triangulation $\cT$ with $t$ tetrahedra, then each coordinate of $v(S) \in \Z^{7t}$ is at most $2^{7t-1}$.
\end{theorem}

The key fact we use to detect torus knots, cabled knots, and composite knots, proved originally by Simon \cite{simon}, is that a knot $K$ lies in one of these classes if and only if its exterior $E_K = S^3 \ssm N(K)$ contains a properly embedded essential annulus.  The annuli in question split $E_K$ into two connected complements as follows:
\begin{itemize}
\item If $K$ is a torus knot, then the complement is a pair of solid tori.
\item If $K$ is a cable, the complement is a solid torus and a nontrivial knot complement.
\item If $K$ is a composite knot, the complement is a pair of nontrivial knot complements.
\end{itemize}
By \cite[Lemma~15.26]{burde-zieschang-book}, if $K$ is a torus knot or a cable then every essential annulus is a cabling annulus: this means that there is a knot $K' \subset S^3$ such that $K$ is a nonseparating curve in the torus $T = \partial N(K')$ (note that $K'$ is unknotted if $K$ is a torus knot), and the annulus is then the intersection $T \cap E_K$.  Such annuli decompose $E_K$ into the exterior $E_{K'}$ and the solid torus $N(K')$.  If instead $K$ is a connected sum, then every essential annulus has meridional boundary slope and thus decomposes $E_K$ into a pair of nontrivial knot complements.

Essential annuli in the exterior of $K$ are always isotopic to normal surfaces in any given triangulation of $E_K$, by the following theorem; we take the statement from Jaco and Rubinstein \cite[Theorem~2.4]{jaco-rubinstein}, who attribute it to Haken \cite{haken-results}.

\begin{theorem} \label{thm:isotopic-to-normal}
Let $S$ be a properly embedded, essential surface in an irreducible 3-manifold $Y$ with incompressible boundary.  If $\cT$ is any triangulation of $Y$ and $\cT^{(1)}$ its 1-skeleton, then $S$ is isotopic to a normal surface $S'$ of \emph{least weight}, meaning that $\#(S' \cap \cT^{(1)})$ is minimized among normal surfaces in its isotopy class.
\end{theorem}

Theorem~\ref{thm:isotopic-to-normal} is true for arbitrary essential surfaces, but for annuli we know significantly more: by work of Jaco and Tollefson \cite{jaco-tollefson}, we can find essential annuli as above which are not just normal but vertex surfaces.  The following result holds for annuli and tori in any orientable, compact, irreducible 3-manifold with incompressible boundary, though we only need it for annuli in knot complements.

\begin{theorem}[{\cite[Corollary~6.8]{jaco-tollefson}}] \label{thm:annulus-carrier}
Let $A$ be a normal, two-sided, essential annulus of least weight in the exterior $E_K$ of a nontrivial knot $K \subset S^3$.  Then every vertex surface in the carrier of $A$ is either an essential annulus or an essential torus.
\end{theorem}

We combine all of these results to guarantee the existence of essential annuli with bounded weight in the complements of torus knots, cables, and connected sums of knots.

\begin{proposition} \label{prop:small-normal-annuli}
Let $K \subset S^3$ be a nontrivial knot, and let $\cT$ be a triangulation of its exterior $E_K = S^3 \ssm N(K)$, with $t$ tetrahedra.  If $K$ is a torus knot, cable knot, or composite knot, then there is an essential annulus $A \subset E_K$ which is normal with respect to $\cT$ of total weight
\[ ||v(A)||_{L^1} \leq 7t \cdot 2^{7t}. \]
The complement $E_K \ssm N(A)$ has two connected components, which are either a pair of solid tori, a solid torus and a nontrivial knot complement, or a pair of nontrivial knot complements depending on whether $K$ is a torus, a cable knot, or a composite knot respectively.
\end{proposition}

\begin{proof}
By hypothesis $E_K$ contains an essential annulus, which is isotopic to a normal annulus $A_0$ of least weight by Theorem~\ref{thm:isotopic-to-normal}.  Theorem~\ref{thm:annulus-carrier} says that every vertex surface in $\cC(A_0)$ is an essential annulus or an essential torus, and since $v(A_0)$ is a linear combination of the normal coordinates of these vertex surfaces, they cannot all be tori or else $A_0$ would not have nonempty boundary in $\partial E_K$.

We conclude that there is an essential normal annulus $A$ which is a vertex surface, and since it is connected it must be either minimal or twice a minimal surface (which is then nonorientable, hence a M\"obius band).  The bound of Theorem~\ref{thm:hlp-bound} says that each of the $7t$ coordinates of $v(A)$ is then at most $2^{7t}$, so that $||v(A)||_{L^1} \leq 7t \cdot 2^{7t}$ as claimed.  Finally, the description of $E_K \ssm N(A)$ follows from \cite[Lemma~15.26]{burde-zieschang-book} as described above.
\end{proof}

Given normal surface coordinates for a properly embedded annulus $A$, the following lemma will help us verify that $A$ is incompressible.

\begin{lemma} \label{lem:annulus-pi1}
Let $A$ be a properly embedded annulus in the exterior $E_K = S^3 \ssm N(K)$ of a nontrivial knot.  The map $\pi_1(A) \to \pi_1(E_K)$ is injective if and only if each component of $\partial A$ represents a nontrivial class in $H_1(\partial E_K)$.
\end{lemma}

\begin{proof}
We note that the two components of $\partial A$ represent the same class in $\pi_1(E_K)$, since $A$ provides a homotopy between them.  The map $\pi_1(\partial E_K) \to \pi_1(E_K)$ is injective since $K$ is nontrivial, so the components must then also be identical in $\pi_1(\partial E_K) \cong H_1(\partial E_K)$.

Let $\gamma$ be a component of $\partial A$.  Then $\gamma$ generates $\pi_1(A) \cong \Z$, so every class in $\pi_1(A)$ is represented by a curve $\gamma^n$ in $\partial E_K$, and thus the map $\pi_1(A) \to \pi_1(E_K)$ factors as $\langle\gamma\rangle \to \pi_1(\partial E_K) \to \pi_1(E_K)$.  Since the map $\pi_1(\partial E_K) \to \pi_1(E_K)$ is injective,  we see that $\pi_1(A) \to \pi_1(E_K)$ is injective if and only if $\langle\gamma\rangle \to \pi_1(\partial E_K)$ is as well.  But $\pi_1(\partial E_K) \cong H_1(\partial E_K)$ is torsion-free, so the latter map is injective if and only if $[\gamma]$ is a nonzero class in $H_1(\partial E_K)$.
\end{proof}

We explain how to check that the criteria of Lemma~\ref{lem:annulus-pi1} are satisfied.  If $E_K$ has a triangulation $\cT$ restricting to a one-vertex triangulation of the boundary torus, and $A$ is a properly embedded normal annulus $A$ in $E_K$, then $\partial A$ is naturally described as a normal curve in the surface $\partial E_K$ with respect to $\cT|_{\partial E_K}$.  Normal curves in this triangulation were studied by Jaco and Sedgwick \cite[Theorem~3.6]{jaco-sedgwick}; they are described by three nonnegative integers per triangle, counting the number of arcs around each vertex.
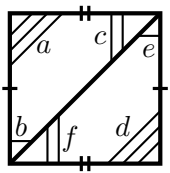
\[
\begin{tikzpicture}
  \draw[ultra thick] (0,0) rectangle (2,2);
  \draw[ultra thick] (0,0) -- (2,2);
  \foreach \x in {0,2} {
      \draw[very thick] ($(\x-0.1,1)$) -- ($(\x+0.1,1)$);
  }
  \foreach \x in {0.95,1.05} {
      \foreach \y in {0,2} {
          \draw[very thick] ($(\x,\y-0.1)$) -- ($(\x,\y+0.1)$);
      }
  }
  \foreach \x in {0,1,2} {
      \draw ($(0,1.7-0.2*\x)$) -- ($(0.3+0.2*\x,2)$);
      \draw ($(1.7-0.2*\x,0)$) -- ($(2,0.3+0.2*\x)$);
  }
  \foreach \x in {0.5,0.65} {
      \draw (\x,\x) -- (\x,0);
      \draw ($(2-\x,2-\x)$) -- ($(2-\x,2)$);
   }
  \draw (0,0.3) -- (0.3,0.3);
  \draw (1.7,1.7) -- (2,1.7);
  \node at (0.45,1.55) {$a$};
  \node at (0.15, 0.5) {$b$};
  \node at (1.2, 1.675) {$c$};
  \node at (1.5, 0.5) {$d$};
  \node at (1.85, 1.5) {$e$};
  \node at (0.8, 0.325) {$f$};
\end{tikzpicture}
\]
A normal curve satisfies one matching equation per edge of the triangulation, namely
\[ a+b=d+e, \qquad a+c=d+f, \qquad b+c=e+f. \]
These are equivalent to $a=d$, $b=e$, and $c=f$, so the curve is uniquely characterized by the triple $(a,b,c)$.  Two curves in this triangulation are normally isotopic if and only if they are isotopic \cite[Lemma~3.5]{jaco-sedgwick}, so if $A$ is a normal annulus with boundary $\gamma_1 \cup \gamma_2$, then $\gamma_1$ is isotopic to $\gamma_2$ and so $\partial A$ has normal coordinates of the form $(2a,2b,2c)$.  If the $\gamma_i$ are homologically nontrivial, then they each represent a primitive class since they are embedded, and hence each has odd intersection with some element of a basis of $H_1(\partial E_K;\Z)$.  In particular, they are homologically nontrivial if and only if $\partial A$ has normal coordinates $(2a,2b,2c)$ with at least one of $a+b$ and $a+c$ being odd.

\section{Torus knot recognition is in NP} \label{sec:torus-knot-np}

In this section we show how to certify that a knot diagram $D$ represents a torus knot.  We do not declare which torus knot it is, but this is unnecessary, since Proposition~\ref{prop:which-torus-knot} allows us to determine this in polynomial time.  Our certification relies on the fact that a knot $K \subset S^3$ is a torus knot if and only if there is an essential annulus $A$ in the exterior of $K$ whose complement is a pair of solid tori; the annulus is the intersection of the exterior with a Heegaard torus on which $K$ lies.

Our certification is a straightforward combination of algorithms by others.  After triangulating the knot complement, we present the annulus $A$ as a normal surface with bounded weight, using Proposition~\ref{prop:small-normal-annuli}; then we cut the knot complement along $A$ and triangulate the remaining pieces using work of Lackenby \cite{lackenby-knottedness}.  Finally, we certify that $A$ is essential and that the remaining pieces are indeed solid tori using an algorithm of Ivanov \cite{ivanov}, and this verifies that we have a torus knot.

We now begin to construct the certificate in detail, by producing a small triangulation of the knot exterior $E_K = S^3 \ssm N(K)$.  We assume that every knot diagram has at least three crossings, since otherwise it is clearly a diagram of the unknot.

\begin{proposition} \label{prop:triangulate-exterior}
Let $D$ be a diagram with $n$ crossings of a knot $K$.  There is a universal constant $C$ and an algorithm which produces in $\poly(n)$ time a triangulation $\cT$ of the exterior $E_K$, such that $\cT$ has at most $Cn$ tetrahedra and its restriction to the boundary torus $\partial E_K$ consists of two triangles and a single vertex.
\end{proposition}

\begin{proof}
Hass, Lagarias, and Pippenger \cite[Lemma~7.1]{hass-lagarias-pippenger} show how to triangulate $S^3$ with at most $440n+2$ tetrahedra in $O(n\log n)$ time so that $K$ lies in the 1-skeleton, and then in \cite[Lemma~7.2]{hass-lagarias-pippenger} they apply two baryocentric subdivisions and remove a neighborhood of $K$ to get a triangulation $\cT_0$ of $E_K$ with $t \leq 24^2(440n+2)$ tetrahedra.  The restriction of $\cT_0$ to $\partial E_K$ may be complicated, but following an algorithm of Lackenby \cite[Proposition~10.3]{lackenby-knottedness} we can attach at most $4t$ tetrahedra to $\partial E_K$ in $\poly(t) = \poly(n)$ time to reduce the number of boundary vertices to one.  The resulting triangulation $\cT$ of $E_K$ has at most $5t \leq 5\cdot 24^2(440n+2)$ tetrahedra, so we take $C=5\cdot 24^2 \cdot 441$.
\end{proof}

Supposing that $\cT$ has $t$ tetrahedra, Proposition~\ref{prop:small-normal-annuli} provides an essential annulus $A \subset E_K$ which is normal of weight $||v(A)||_{L^1} \leq 7t \cdot 2^{7t}$, and which splits $E_K$ into a pair of solid tori.  In order to triangulate the complement of $A$ in $E_K$, we use the following proposition which is due to, but not explicitly stated by, Lackenby \cite{lackenby-knottedness}.

\begin{proposition} \label{prop:annulus-triangulation}
Let $M$ be a compact, irreducible, orientable 3-manifold with incompressible boundary, and let $\cT$ be a triangulation of $M$ with $t$ tetrahedra.  If $A$ is a properly embedded normal annulus in $M$ with weight $w = ||v(A)||_{L^1}$, then there is an algorithm which builds a triangulation of $M' = M \ssm N(A)$ with at most $200t$ tetrahedra in $\poly(t\log(w))$ time.
\end{proposition}

\begin{proof}
We repeat the proof of \cite[Theorem~11.4]{lackenby-knottedness} verbatim.  The algorithm converts $\cT$ into a handle structure, cuts along $A$ and converts the resulting handle structure back into a triangulation.  This may produce at least $O(w)$ handles, but most of them are quite simple: they are 3-balls cut out by two parallel elementary disks of $A$ inside a single handle of $\cT$, and these fit into a union $\cB$ of $I$-bundles over subsurfaces of $A \sqcup A$, called ``parallelity bundles.''  So we triangulate the linearly many (in $t$) handles of $M' \ssm \cB$, determine the topology of each component of $\cB$ by \cite[Theorem~9.3]{lackenby-knottedness}, and then use this to build a simpler triangulation of $\cB$ by hand which we glue to the triangulation of $M' \ssm \cB$.

The only difference is that \cite[Theorem~11.4]{lackenby-knottedness} is stated for $A$ a union of tori, and so the base of each parallelity bundle has genus at most 1.  In our setting the genus is zero since $A$ is an annulus, so the same estimates as in Lackenby's original proof still hold.
\end{proof}

We combine the above results to conclude that \torusknot{} is in \NP{}.

\begin{theorem} \label{thm:torus-knot-np}
Given a diagram $D$ with $n$ crossings of a knot $K$, if $K$ is a torus knot then there is a certificate which can be used to verify this in $\poly(n)$ time.
\end{theorem}

\begin{proof}
Using Proposition~\ref{prop:triangulate-exterior}, we produce a triangulation $\cT$ of the exterior $E_K = S^3 \ssm N(K)$ with $t = O(n)$ tetrahedra.  Given a properly embedded normal annulus $A \subset E_K$, we can use Proposition~\ref{prop:annulus-triangulation} to triangulate its complement with $O(n)$ tetrahedra in polynomial time as well.  Our certificate therefore consists of:
\begin{enumerate}
\item Normal coordinates $v(A)$ for an essential annulus $A$ in $E_K$, with $||v(A)|| \leq 7t\cdot2^{7t}$;
\item Certificates that the two triangulated components of $E_K \ssm N(A)$ are both solid tori, as provided by Ivanov \cite[Theorem~3]{ivanov}.
\end{enumerate}
\noindent To verify the certificate, we do the following:
\begin{enumerate}
\item Verify that $\Delta_K(t) \neq 1$, as described in Proposition~\ref{prop:compute-signature}. \label{i:alexander-not-1}
\item Verify that $A$ is an annulus and that the components of $\partial A$ are homologically nontrivial in $\partial E_K$. \label{i:annulus-injective}
\item Apply Proposition~\ref{prop:annulus-triangulation} to triangulate $E_K \ssm N(A)$ and verify that it has two connected components, say $M_1$ and $M_2$.
\item Verify the certificates that the given triangulations of $M_1$ and $M_2$ produce solid tori, following \cite{ivanov}.
\end{enumerate}
Step~\eqref{i:alexander-not-1} ensures that $K$ is not the unknot, and hence it will help us show that $A$ is essential.  Indeed, if $K$ is knotted then step~\eqref{i:annulus-injective} will show that the map $\pi_1(A) \to \pi_1(E_K)$ is injective by Lemma~\ref{lem:annulus-pi1}, so it is either essential or boundary-parallel.  But if $A$ is boundary-parallel then one component of its complement is homeomorphic to $E_K$, which can only be a solid torus if $K$ is the unknot.  Therefore, once we know that $E_K \ssm N(A)$ is a disjoint union of two solid tori, we can conclude that $A$ is essential and $K$ is indeed a torus knot.  Moreover, torus knots do not have Alexander polynomial 1, so step~\eqref{i:alexander-not-1} will not incorrectly eliminate any torus knots.

To verify that $A$ is an annulus, we first check that the vector $v(A)$ solves the normal surface equations and that it produces a properly embedded surface of Euler characteristic zero with nonempty boundary.  We can then use Agol-Hass-Thurston's orbit-counting algorithm to check in $\poly(t\log ||v(A)||_{L^1})$ time that $A$ is connected \cite[Corollary~14]{agol-hass-thurston}, and we check that it is orientable by using the same algorithm to see that the surface with normal coordinates $2v(A)$ is not connected, as in \cite[Proof of Theorem~2]{agol-hass-thurston}.  Since $A$ is connected and orientable with $\chi(A) = 0$ and $\partial A \neq \emptyset$, it must be an annulus, and this verification takes a total of $\poly(t \log ||v(A)||_{L^1})$ time.  Similarly, we can verify that the components of $\partial A$ are homologically nontrivial as described following Lemma~\ref{lem:annulus-pi1} in $O(t + \log ||v(A)||_{L^1})$ time, since we can read off the normal components of $\partial A$ immediately once we identify the tetrahedra adjacent to $\partial E_K$.

To see that each of these steps can be done in $\poly(n)$ time, we now observe that $t \leq Cn$, where $C$ is a universal constant, and that $||v(A)||_{L^1} \leq 7t \cdot 2^{7t}$ by Proposition~\ref{prop:small-normal-annuli}.  Thus $\poly(t \log ||v(A)||_{L^1}) = \poly(n)$, and we can construct the triangulation of $M_1 \sqcup M_2$, which has at most $200t = O(n)$ tetrahedra, in $\poly(n)$ time.  Finally, since $M_1$ and $M_2$ each have $O(n)$ tetrahedra we can verify their solid torus certificates in $\poly(n)$ time as well.
\end{proof}

\section{Certifying that hyperbolic knots are not torus knots} \label{sec:hyperbolic-certificates}

Let $D$ be a diagram with $n$ crossings which represents a hyperbolic knot $K$.  Assuming the generalized Riemann hypothesis, we will provide a certificate, verifiable in $\poly(n)$ time, that $K$ is not a torus knot.  The construction of these certificates follows the same ideas as Kuperberg's certificates for knottedness \cite{kuperberg}.  As in \cite{kuperberg}, we need GRH to show that these certificates exist, but once they are known to exist they can be verified unconditionally.  This will be a key step in the proof of Theorem~\ref{thm:main-torus}, specifically the assertion that $\torusknot \in \coNP{}$, because it suffices to either certify that a given non-torus knot is hyperbolic or certify that it is a satellite knot; we will discuss certificates for satellite knots in Section~\ref{sec:satellite-np}.

The diagram $D$ determines a Wirtinger presentation
\[ \pi_1(S^3 \ssm K) = \langle g_1, \dots, g_n \mid r_1, \dots, r_n \rangle \]
of the knot group, in which each generator $g_i$ is a meridian around some strand and the relations $r_i$ all have the form $g_{m_i}g_{n_i}g_{m_i}^{-1} = g_{p_i}$ for some $m_i,n_i,p_i$.  Using $D$, we fix peripheral elements $\mu = g_1$ and $\lambda$, representing a meridian and a longitude.  We observe that $\lambda$ can be expressed as a product of at most $2n$ of the generators $g_i$ and their inverses: to see this, we use Seifert's algorithm to find a Seifert surface $\Sigma$ consisting of some number of disks connected by $n$ bands (one per crossing), take a parallel copy of $K$ inside $\Sigma$, and observe that as this copy passes through each band twice it contributes a total of two generators (or their inverses) to $\lambda$.

\begin{definition}
An \emph{uncentered certificate} that $D$ is not a diagram of a torus knot consists of:
\begin{enumerate}
\item A pair of relatively prime integers $r,s$ with $|rs| < 3n$ such that $K$ has the same Alexander polynomial and signature as $T_{r,s}$, or $\emptyset$ if no such pair exists;
\item If we have integers $r,s$ instead of $\emptyset$:
\begin{itemize}
\item A prime $p$, with $\log(p) = \poly(n)$;
\item A collection of $2\times 2$ matrices $M_1, M_2, \dots, M_n \in SL_2(\F_p)$, defining a representation
\[ \rho: \pi_1(S^3 \ssm K) \to SL_2(\F_p) \]
by $\rho(g_i) = M_i$, such that $\rho(\mu^{rs}\lambda) \rho(g_i) \neq \rho(g_i)\rho(\mu^{rs}\lambda)$ for some $i$.
\end{itemize}
\end{enumerate}
\end{definition}

\noindent To verify an uncentered certificate, we do the following:
\begin{enumerate}
\item Using Proposition~\ref{prop:which-torus-knot}, check in $\poly(n)$ time that the choice of $(r,s)$ or $\emptyset$ was made correctly.  If we have $\emptyset$, then we stop; otherwise $|rs| < 3n$ by Lemma~\ref{lem:crossing-number-pq}.
\item Verify that $\det(M_i) = 1$ for all $i$, and that the $M_i$ satisfy the relations $r_1,\dots,r_n$.
\item Compute $\rho(\mu^{rs}\lambda)$ and verify that it does not commute with some $M_i$.
\end{enumerate}

It is mostly clear that if the certificate exists, then it can be verified in $\poly(n)$ time.  The only part which requires some additional thought is the computation of $\rho(\mu^{rs}\lambda)$.  Since $\mu^{rs}\lambda$ can be written as a word of length at most $|rs| + 2n < 5n$ in the generators $g_1^{\pm 1}, \dots, g_n^{\pm 1}$, we can express $\rho(\mu^{rs}\lambda)$ as a product of the matrices $M_1^{\pm 1}, \dots, M_n^{\pm 1}$ of length at most $5n$ and so it can also be computed in polynomial time.  We also remark that we do not actually need to verify that $p$ is prime: even if it is not, the representation $\rho$ still certifies that $\mu^{rs}\lambda$ is not central, which is all that matters.

\begin{proposition} \label{prop:torus-knots-sl2fp}
If $D$ has an uncentered certificate, then it is not a diagram of a torus knot.
\end{proposition}

\begin{proof}
Suppose that $D$ is a diagram of $T_{r,s}$.  We can take the standard genus-1 Heegaard splitting of $S^3$ and embed $T_{r,s}$ in the Heegaard torus.  Then $\pi_1(S^3 \ssm T_{r,s})$ is generated by the cores of the two genus-1 handlebodies, and $\mu^{rs}\lambda$ is a curve parallel to $T_{r,s}$ in the Heegaard torus, so it commutes with each generator.  Since $\mu^{rs}\lambda$ is central, its image under any representation $\rho: \pi_1(S^3 \ssm T_{r,s}) \to SL_2(\F_p)$ must commute with everything in the image of $\rho$.
\end{proof}

In contrast, we will prove that hyperbolic knots have such certificates, assuming GRH.

\begin{theorem} \label{thm:hyperbolic-certificates}
Assume the generalized Riemann hypothesis.  If $D$ is a diagram of a hyperbolic knot, then it admits an uncentered certificate.
\end{theorem}

\begin{remark}
It is entirely possible that satellite knots also admit uncentered certificates, but we do not know a proof of this.  In Section~\ref{sec:satellite-np} we will use a different strategy to certify satellite knots.
\end{remark}

\begin{lemma} \label{lem:noncommutative-variety}
Let $D$ be a knot diagram with $n > 0$ crossings representing a knot $K$, and fix an integer $m$ with $|m| < 3n$.  There is an algebraic variety which is nonempty if and only if there is a representation
\[ \rho: \pi_1(S^3 \ssm K) \to SL_2(\C) \]
such that $\rho(\mu^m\lambda)$ does not commute with some other element $\rho(g)$ of the image.  Moreover, we can define such a variety using $8n$ variables and $5n+1$ polynomials, each of which has maximum degree at most $5n+2$ and integer coefficients bounded by $2^{5n+1}$ in absolute value.
\end{lemma}

\begin{proof}
We use $D$ to construct a Wirtinger presentation
\[ \pi_1(S^3 \ssm K) = \langle g_1,g_2,\dots,g_n \mid r_1,r_2,\dots,r_n \rangle, \]
in which each $g_i$ is a meridional loop around a strand of $D$ and each crossing produces a relation $r_i$ of the form $g_{m_i}g_{n_i}g_{m_i}^{-1} = g_{p_i}$.  We will take $\mu =  g_1$ and let $\lambda$ be a longitude.

We can construct the $SL_2(\C)$ representation variety $R(K)$ from this presentation as an algebraic subset of $\C^{4n}$ by using $4n$ generators $a_i,b_i,c_i,d_i$, packaged into matrices $M_i = \twosmallmatrix{a_i}{b_i}{c_i}{d_i}$.  Then $R(K)$ is defined by a total of $5n$ polynomial equations: we require $\det(M_i) = a_id_i - b_ic_i = 1$ for each $i$, and each relation $r_i$ of the form $g_{m_i}g_{n_i}g_{m_i}^{-1} = g_{p_i}$ ($1 \leq i \leq n$) contributes four polynomial equations coming from the entries of the $2\times 2$ matrix equation $M_{m_i}M_{n_i} = M_{p_i}M_{m_i}$.

We must now select only those representations $\rho$ for which $\rho(\mu^{m}\lambda)$ does not commute with some $M_i$, so we first introduce some notation.  As explained above, we can write $\lambda$ as a word of length $l \leq 2n$ in the generators $g_1^{\pm 1}, \dots, g_n^{\pm 1}$, say $\lambda = g_{i_1}^{\epsilon_1} g_{i_2}^{\epsilon_2} \dots g_{i_l}^{\epsilon_l}$ where $\epsilon_j = \pm 1$ for each $j$.  We will let $\epsilon = \sign(m) \in \{\pm 1\}$ and write
\[ A_m = (M_1^{\epsilon})^{|m|} M_{i_1}^{\epsilon_1} M_{i_2}^{\epsilon_2} \dots M_{i_l}^{\epsilon_l}, \]
where for any $j$ we interpret $M_j^{-1}$ as the matrix $\twosmallmatrix{d_j}{-b_j}{-c_j}{a_j}$.  This is a product of at most $|m|+l < 5n$ matrices $M_j^{\pm 1}$, so a simple induction says that its entries are sums of at most $2^{5n-1}$ monomials of degree at most $5n$ each in the generators $a_j,b_j,c_j,d_j$.  Moreover, if $a_j,b_j,c_j,d_j$ correspond to a representation $\rho \in R(K)$ then we have $A_m = \rho(\mu^m\lambda)$.

In order to select only the desired representations, we need only check that $A_m$ does not commute with the image $M_i$ of some generator $g_i$.  We accomplish this by the Rabinowitsch trick, following Kuperberg \cite{kuperberg}: namely, we add another $4n$ generators $t_{ijk}$ with $1 \leq i \leq n$ and $1 \leq j,k \leq 2$, and we add a single equation
\[ \sum_{i,j,k} t_{ijk}\big(A_mM_i - M_iA_m\big)_{j,k} = 1, \]
where $\big(M\big)_{j,k}$ denotes the $(j,k)$th entry of the matrix $M$.  It is clear that this equation is satisfiable for some values of the $t_{ijk}$ if and only if $A_m M_i \neq M_i A_m$ for some $i$, i.e.\ if and only if the corresponding representation $\rho$ satisfies $\rho(\mu^m\lambda)\rho(g_i) \neq \rho(g_i)\rho(\mu^m\lambda)$.  Then every term in this equation has degree at most $5n+2$ and integer coefficients bounded in absolute value by $2^{5n+1}$, since $A_mM_i$ and $M_iA_m$ are each sums of at most $2^{5n}$ monomials of degree at most $5n+1$, and so the lemma follows.
\end{proof}

In order to produce an $SL_2(\F_p)$ representation from the $SL_2(\C)$ representation variety, we use the following theorem, which is a combination of results of Koiran \cite{koiran} and Lagarias-Odlyzko and Weinberger \cite{lagarias-odlyzko,weinberger}.

\begin{theorem}[{\cite[Theorem~3.3]{kuperberg}}] \label{thm:grh-solution-mod-p}
Let $f_1,\dots,f_m \in \Z[x_1,\dots,x_n]$ be non-constant integer polynomials with degree at most $d$ and all coefficients having absolute value at most $r$, and suppose that the system $f_1=f_2=\dots=f_m=0$ has a solution in $\C^n$.  Assuming the generalized Riemann hypothesis, there is a prime $p$ with $\log(p) = \poly(n,m,\log(d),\log(r))$ such that the system has a solution in $(\Z/p\Z)^n$.
\end{theorem}

We remark that the generalized Riemann hypothesis is used in \cite{lagarias-odlyzko,weinberger} to provide an effective version of the Chebotarev density theorem, guaranteeing that we can take $p$ with at most polynomially many digits.  The basic result we need to assume (and which is implied by GRH) is stated as \cite[Theorem~3.2]{kuperberg}; it asserts that an irreducible polynomial $h \in \Z[x]$ with degree $d$ and all coefficients between $-r$ and $r$ has a root in $\Z/p\Z$ for some prime $p = \poly(d,\log r)$.  Without GRH, Koiran's work \cite{koiran} plus the classical Chebotarev density theorem says unconditionally that we can take any $p$ in some positive-density subset of the primes, but it does not provide a small enough bound on the minimal value of $p$.

\begin{proof}[Proof of Theorem~\ref{thm:hyperbolic-certificates}]
We may assume that $K$ has the same Alexander polynomial and signature as the torus knot $T_{r,s}$, since otherwise our certificate does not require an actual $SL_2(\F_p)$ representation.  In this case Lemma~\ref{lem:crossing-number-pq} says that $|rs| < 3n$.

Since $K$ is hyperbolic, there is a discrete, faithful representation $\rho_0: \pi_1(S^3 \ssm K) \to SL_2(\C)$, as proved in \cite[Proposition~3.1.1]{culler-shalen-splittings} and attributed to Thurston.  A theorem of Burde and Zieschang \cite{burde-zieschang} says that $\pi_1(S^3 \ssm K)$ has trivial center since $K$ is not a torus knot, so $\mu^{rs}\lambda$ does not commute with some element $g$ of the knot group.  But then $\rho_0(\mu^{rs}\lambda)$ does not commute with $\rho_0(g)$ since $\rho_0$ is faithful, so if we set $m=rs$ then the variety of Lemma~\ref{lem:noncommutative-variety} is nonempty.  We apply Theorem~\ref{thm:grh-solution-mod-p} to conclude that the system of equations defining this variety has a solution mod $p$, where
\[ \log(p) = \poly(8n, 5n+1, \log(5n+2), \log(2^{5n+1})) = \poly(n). \]
This solution gives the desired prime $p$ and representation $\pi_1(S^3 \ssm K) \to SL_2(\F_p)$.
\end{proof}

\section{Satellite knot recognition is in \NP} \label{sec:satellite-np}

A knot $K \subset S^3$ is a satellite knot if and only if its exterior contains some incompressible, non-boundary-parallel tori.  In this case the JSJ decomposition \cite{jaco-shalen, johansson} of the knot exterior is nontrivial, with $E_K = S^3 \ssm N(K)$ being cut along such incompressible tori into atoroidal and Seifert fibered pieces.  In particular, we can certify that a $K$ is a satellite knot by providing the JSJ tori and efficiently checking that they are incompressible and not all boundary-parallel.

In the first part of this certification, we specify the JSJ tori as a union $T$ of normal surfaces, using exponential bounds on their weight due to Mijatovi\'c \cite{mijatovic}, and then use Lackenby's work \cite{lackenby-knottedness} to certify that they are incompressible and to triangulate their complement.  The incompressibility certificate is a crucial part of Lackenby's Thurston norm certificate \cite[Section~13]{lackenby-knottedness}, which we explain very briefly here.

The complement $E_K \ssm N(T)$ can be written $M'_1 \sqcup M'_2$, where $M'_1$ consists of the atoroidal components and $M'_2$ the Seifert fibered components.  We certify using \cite[Theorem~12.3]{lackenby-knottedness} that $M'_2$ has incompressible boundary.  We then turn $M'_1$ into a sutured manifold with empty sutures and certify that $\partial M'_1$ is incompressible by providing a certain sutured manifold hierarchy for $(M'_1,\emptyset)$, of length linear in the number of tetrahedra.  The certificate thus includes handle structures for every other manifold in the hierarchy, denoted $(M_i,\gamma_i)$, and normal surface vectors for the decomposing surfaces $S_i \subset M_i$.  (The remaining sutured manifolds are decomposed along annuli which can be determined algorithmically; see \cite[Theorem~10.1]{lackenby-knottedness}.)  The certificate also includes proof that the last manifold in the hierarchy is a product, using the theorem of Schleimer \cite{schleimer} and Ivanov \cite{ivanov} that 3-ball recognition is in \NP{}.  Much of the difficulty comes from needing to decompose each $(M_i,\gamma_i)$ efficiently along the corresponding $S_i$ and subsequent annuli, so that a handle structure on the resulting $(M_{i+1},\gamma_{i+1})$ can be constructed and verified in polynomial time.

Taking the incompressibility of the JSJ tori for granted at the moment, we can understand the components of their complement as follows.

\begin{lemma} \label{lem:jsj-complement}
Let $K \subset S^3$ be a knot with exterior $E_K = S^3 \ssm N(K)$, and let $T \subset E_K$ be a union of finitely many disjoint incompressible tori.  Then one component of $E_K \ssm N(T)$ is the complement of some knot $K' \subset S^3$, and every other component is the complement of a knot in $S^1 \times D^2$.
\end{lemma}

\begin{proof}
If $T$ consists of a single torus, then we note that $T$ bounds a solid torus $S^1 \times D^2$ in $S^3$, and since $T$ is incompressible in $E_K$ it follows that $K$ must lie in this solid torus.  If we let $K' \subset S^3$ be the core of this solid torus, then $E_K \ssm N(T)$ consists of two components, one of which is the exterior of $K'$ and the other of which is $(S^1 \times D^2) \ssm N(K)$.

Now suppose that $T$ consists of $n$ tori $T_1, \dots, T_n$, where $n \geq 2$.  As above, each $T_i$ separates $E_K$ into a knot exterior $E_{K_i}$, with boundary $T_i$, and $(S^1 \times D^2) \ssm N(K)$.  We define a total ordering of the tori by $T_i < T_j$ if $T_i \subset E_{K_j}$, and we relabel the tori so that $T_1 < T_2 < \dots < T_n$.  Then $T_n$ divides $S^3 \ssm N(K)$ into two components, one of which is $(S^1 \times D^2) \ssm N(K)$.  The other component is the exterior $E_{K_n}$, and it contains the incompressible tori $T_1,\dots,T_{n-1}$, so the lemma follows by induction on $n$.
\end{proof}

The component which is a knot complement cannot be an unknot complement since its boundary is incompressible, so in order to see that we have a satellite, we just need to show that one of the remaining components is not $T^2 \times I$.  Since these components are all complements of knots in $S^1 \times D^2$, we use the following criterion, originally stated as Theorem~\ref{thm:main-solid-torus-rep}.

\begin{theorem} \label{thm:solid-torus-knot-rep}
Let $P \subset S^1 \times D^2$ be a knot, and let $E_P = (S^1 \times D^2) \ssm N(P)$ be its exterior.  There is a representation
\[ \pi_1(E_P) \to SL_2(\C) \]
with nonabelian image if and only if $P$ is not isotopic to the core $S^1 \times \{0\}$.
\end{theorem}

\begin{proof}
If $P$ is a core of the solid torus then $\pi_1(E_P) = \Z^2$ is abelian, so the claim follows immediately.  If $P$ lies in a 3-ball then we can write $E_P = (S^1\times D^2) \# (S^3 \ssm N(P))$, whose fundamental group surjects onto $\Z * H_1(S^3 \ssm P) = \Z * \Z$.  This has a nonabelian representation for any pair of noncommuting elements of $SL_2(\C)$, so we can assume from now on that $P$ does not lie in a 3-ball.  We let $M = \{\pt\} \times \partial D^2$ and $L = S^1 \times \{\pt\}$ be generators of $H_1(S^1 \times \partial D^2)$, and we let $\mu, \lambda \in \partial N(P)$ denote a meridian and a longitude of $P$ respectively.  If $P$ has winding number $w$, then these satisfy $[M] = w[\mu]$ and $[\lambda] = w[L]$ in $H_1(E_P; \Z)$, which is generated by $[\mu]$ and $[L]$.

We now fix an integer $r \neq 0$ and construct 3-manifolds $Y_n$ for any $n\in\Z$ by Dehn filling the exterior of $P$ along a pair of curves: we fill $\partial N(P)$ along the slope $\mu+r\lambda$, and then we fill $S^1 \times \partial D^2$ along $nM + (nrw^2+1)L$.  The filling curves belong to the homology classes $[\mu]+rw[L]$ and $nw[\mu]+(nrw^2+1)[L]$, which span all of $H_1(E_p)$, so $Y_n$ is a homology sphere.  If $Y_n$ is not homeomorphic to $S^3$, then Zentner \cite{zentner} proved that $\pi_1(Y_n)$ admits an irreducible $SL_2(\C)$ representation, and since $\pi_1(Y_n)$ is a quotient of $\pi_1(E_P)$, the composition
\[ \pi_1(E_P) \to \pi_1(Y_n) \to SL_2(\C) \]
is the desired representation.

If we did not get a representation out of this construction, then we must have $Y_n \cong S^3$ for all $n$.  
Letting $K \subset Y_0 \cong S^3$ denote the core of the filling along $L$, we note that the nontrivial surgery on $K$ corresponding to filling along $M+(rw^2+1)L$ produces $Y_1 \cong S^3$.  Gordon and Luecke's solution to the knot complement problem \cite{gordon-luecke-complement} therefore says that $K$ is the unknot.  Its complement is a solid torus which we constructed by filling $E_P$ along the curve $\mu+r\lambda \subset \partial N(P)$, so $P$ has a nontrivial $S^1 \times D^2$ surgery of slope $\mu+r\lambda$.  In other words, $P$ must be a Berge-Gabai knot \cite{berge,gabai-tori}, meaning it is either a torus knot (i.e.,\ isotopic into $S^1 \times \partial D^2$) or a 1-bridge braid.

Since we chose $r$ arbitrarily at the beginning, the above argument shows that if there are no nonabelian representations $\pi_1(E_P) \to SL_2(\C)$ then any $\mu+r\lambda$ is an $S^1\times D^2$ surgery slope for $P$.  But Berge \cite{berge} and Gabai \cite{gabai-1bridge} showed that 1-bridge braids have at most two nontrivial $S^1 \times D^2$ surgeries, so $P$ must be isotopic into $S^1 \times \partial D^2$.  Suppose that $P$ represents the class $p[M]+q[L]$ for some coprime integers $p,q$ with $q \geq 2$, since if $q=0$ then $P$ lies in a 3-ball and if $q=1$ then $P$ is isotopic to a core.  If $|p|\geq 2$ as well then we can Dehn fill $E_P$ along the curves $L$ and $\mu+\lambda$ to get $1$-surgery on the $(p,q)$ torus knot in $S^3$.  Then $\pi_1(E_P)$ surjects onto $\pi_1(S^3_1(T_{p,q}))$, which admits a nonabelian $SU(2) \subset SL_2(\C)$ representation by \cite[Theorem~1]{km-su2}, hence $\pi_1(E_P)$ does as well.

The only remaining case is the torus knot with $(p,q) = (1, q)$, whose exterior has fundamental group
\[ \pi_1(E_P) = \Z^2 \ast_{(1,q) \sim q} \Z = \langle x,y,t \mid xy=yx, xy^q=t^q \rangle, \]
as can be seen by splitting $E_P$ along an essential annulus (whose core is parallel to $P$) into $T^2 \times I$ and $S^1 \times D^2$; and its mirror, which has $(p,q) = (-1,q)$ and the same fundamental group.  We use the second relation above to write $x=t^qy^{-q}$, and then substitute this into $xy=yx$ to get $t^qy^{1-q} = yt^qy^{-q}$, or equivalently
\[ \pi_1(E_P) = \langle y,t \mid t^q y = yt^q \rangle. \]
This group has an $SL_2(\C)$ representation defined by $t \mapsto \twosmallmatrix{e^{i\pi/q}}{0}{0}{e^{-i\pi/q}}$, $y \mapsto \twosmallmatrix{0}{1}{-1}{0}$, and this is nonabelian since $q \geq 2$, so this completes the proof.
\end{proof}

\begin{corollary} \label{cor:solid-torus-knot-mod-p}
Let $P \subset S^1 \times D^2$ be a knot whose exterior $E_P = (S^1 \times D^2) \ssm N(P)$ can be triangulated with $t$ tetrahedra.  Assuming the generalized Riemann hypothesis, there is a prime $p$ with $\poly(t)$ digits and a representation
\[ \pi_1(E_P) \to SL_2(\F_p) \]
with nonabelian image if and only if $P$ is not isotopic to the core $S^1 \times \{0\}$.
\end{corollary}

\begin{proof}
If $P$ is isotopic to the core then again $E_P$ has abelian fundamental group, so no such representation exists.  Otherwise there is a presentation of $\pi_1(E_P)$ with $O(t)$ generators and relations, since these are determined by the 1-skeleton and 2-dimensional faces of the triangulation respectively, and moreover each relation can be taken to have polynomial length in $t$.  Then Theorem~\ref{thm:solid-torus-knot-rep} says that there is a representation $\pi_1(E_P) \to SL_2(\C)$ with nonabelian image, so \cite[Theorem~3.4]{kuperberg} guarantees the existence of such an $SL_2(\F_p)$ representation as well, by using the Rabinowitsch trick to discard the representations with abelian image and then applying Theorem~\ref{thm:grh-solution-mod-p}.
\end{proof}

We can now prove Theorem~\ref{thm:main-satellite}, which asserts that \satelliteknot{} is in \NP{}, assuming GRH.  Again we remark that GRH is only required for the existence of a certificate; any given certificate can be unconditionally verified in polynomial time.

\begin{theorem} \label{thm:satellite-np}
Assume the generalized Riemann hypothesis.  Given a diagram $D$ of a satellite knot $K$ with $n$ crossings, there is a certificate which can be used to verify in $\poly(n)$ time that $K$ is a satellite knot.
\end{theorem}

\begin{proof}
We triangulate the exterior $E_K = S^3 \ssm N(K)$ with $t = O(n)$ tetrahedra using Proposition~\ref{prop:triangulate-exterior}.  Then the JSJ tori $T$ of $E_K$ can be realized as normal surfaces with total weight $||v(T)||_{L^1} \leq 2^{80t^2}$ by \cite[Proposition~2.4]{mijatovic}, so by \cite[Theorem~11.4]{lackenby-knottedness} we can triangulate $E_K \ssm N(T)$ with at most $200t = O(n)$ tetrahedra in $\poly(t \log ||v(T)||_{L^1}) = \poly(n)$ time.

If $K$ is a nontrivial satellite, then by Lemma~\ref{lem:jsj-complement} some component $C$ of $E_K \ssm N(T)$ will be the complement of a knot in $S^1 \times D^2$ which is not isotopic to a core, so by Corollary~\ref{cor:solid-torus-knot-mod-p} there will be a prime $p$ with $\log(p) = \poly(n)$ and a representation $\pi_1(C) \to SL_2(\F_p)$ with nonabelian image.  We can determine a presentation
\[ \pi_1(C) = \langle g_1, \dots, g_k \mid r_1, \dots, r_m \rangle \]
in $\poly(t)=\poly(n)$ time from the 2-skeleton of the given triangulation of $C$.  Our certificate therefore consists of:
\begin{enumerate}
\item The normal coordinates $v(T)$ for the JSJ tori of $E_K$;
\item A certificate that $T$ is incompressible, from \cite{lackenby-knottedness};
\item A choice of component $C$ of $E_K \ssm T$;
\item A prime number $p$ with $\poly(n)$ digits;
\item A representation $\rho: \pi_1(C) \to SL_2(\F_p)$ with nonabelian image, specified as a list of $k$ matrices $M_i = \rho(g_i) \in SL_2(\F_p)$.
\end{enumerate}
To verify the certificate, we do the following:
\begin{enumerate}
\item Verify that $T$ is a union of tori, by using the work of Agol-Hass-Thurston \cite[Corollary~17]{agol-hass-thurston}.
\item Verify the certificate that asserts the incompressibility of $T$, as in \cite{lackenby-knottedness}.
\item Verify that the boundary of $C$ has two connected components.
\item Verify that the matrices $M_i$ satisfy $\det(M_i)=1$ and each of the relations $r_j$.
\item Verify that $M_iM_j \neq M_jM_i$ for some $i \neq j$.
\end{enumerate}

This verification can clearly be done in polynomial time; we need only observe that for the first step, the algorithm of \cite[Corollary~17]{agol-hass-thurston} requires time $\poly(t \log ||v(T)||_{L^1}) = \poly(n)$.  If successful, it shows that the specified component $C$ of $E_K \ssm T$ has boundary a pair of incompressible tori but is not $T^2 \times I$, since the representation $\rho$ shows that $\pi_1(C)$ is nonabelian.  This implies that $E_K$ contains an incompressible torus which is not boundary-parallel, and so $K$ is indeed a satellite knot.
\end{proof}

Theorem~\ref{thm:satellite-np} finally allows us to complete the proof of Theorem~\ref{thm:main-torus}.

\begin{proof}[Proof of Theorem~\ref{thm:main-torus}]
We already showed in Theorem~\ref{thm:torus-knot-np} that $\torusknot \in \NP$, so we wish to show that it is in \coNP{}, assuming GRH.  Given a diagram of a knot $K$ which is not a torus knot, we know that $K$ is either the unknot, a satellite knot, or a hyperbolic knot.  If it is an unknot or a satellite knot then we can certify this using Hass-Lagarias-Pippenger's unknottedness certificates \cite{hass-lagarias-pippenger} or Theorem~\ref{thm:satellite-np} respectively.  Otherwise, $K$ is a hyperbolic knot, and in this case we can certify that it is not a torus knot by using an uncentered certificate, whose existence is guaranteed by Theorem~\ref{thm:hyperbolic-certificates}.

As for the $T_{r,s}$ recognition problem, given a knot diagram we can certify that it represents $T_{r,s}$ by first certifying that it is a torus knot, and then verifying in polynomial time (via Proposition~\ref{prop:which-torus-knot}) that it has the same Alexander polynomial and signature as $T_{r,s}$.  If instead it does not represent $T_{r,s}$, then we can certify this by either certifying that it is not a torus knot or by checking that it does not have the same Alexander polynomial and signature as $T_{r,s}$.  Assuming GRH, these certificates exist and can be verified in polynomial time since $\torusknot \in \NP \cap \coNP$, so this establishes that the $T_{r,s}$ recognition problem is in $\NP \cap \coNP$ as well.
\end{proof}

\section{Cable recognition and composite knot recognition are in \NP{}} \label{sec:cable-composite}

In this section we produce certificates, verifiable in polynomial time, which prove that a knot diagram represents a nontrivial cable or composite knot.  The arguments are very similar to our proof in Section~\ref{sec:torus-knot-np} that $\torusknot \in \NP$, in that we use Proposition~\ref{prop:small-normal-annuli} to provide an essential normal annulus and check the components of its complement.

More precisely, if $K$ is cabled or composite then an essential annulus separates the exterior $E_K = S^3 \ssm N(K)$ into two components. Then $K$ is cabled if and only if one component is a solid torus and the other is a nontrivial knot complement $E_{K'}$, where $K'$ is the companion knot; and $K$ is composite if and only if both components are nontrivial knot complements, say $E_{K_1}$ and $E_{K_2}$ where $K = K_1 \# K_2$.  In either case we can recognize solid tori using Ivanov's work \cite{ivanov}, and we can recognize nontrivial knot complements by certifying their Thurston norm as done by Lackenby \cite{lackenby-knottedness}.

In order to check that the Thurston norm of some triangulated knot complement is nonzero, we first need to find a simplicial 1-cocycle with bounded coefficients which represents a nonzero cohomology class.

\begin{proposition} \label{prop:build-cocycle}
Let $Y$ be a compact, connected 3-manifold with $H^1(Y;\Z) = \Z$ and possibly nonempty boundary, and let $\cT$ be a triangulation of $Y$ with $t$ tetrahedra.  Then there is a simplicial 1-cocycle $\phi$, with integer coefficients in the basis of $C^1(Y;\Z)$ dual to the edges of $\cT$, such that $||\phi||_{L^1} \leq \frac{1}{3}(18t)^{6t}$ and the class $[\phi]$ is a generator of $H^1(Y;\Z)$.
\end{proposition}

\begin{proof}
We write the simplicial cochain complex corresponding to $\cT$ as 
\[ C^0(Y;\Z) \xrightarrow{\delta^0} C^1(Y;\Z) \xrightarrow{\delta^1} C^2(Y;\Z) \xrightarrow{\delta^2} C^3(Y;\Z). \]
Labeling the edges of $\cT$ as $e_1, \dots, e_n$, where $n \leq 6t$, the dual elements $e_1^*,\dots,e_n^*$ defined by $e_i^*(e_j) = \delta_{ij}$ generate $C^1(Y;\Z) \cong \Z^n$.  Likewise we let $v_1,\dots,v_m$ and $f_1,\dots,f_k$ be the vertices and faces of $\cT$, with $k,m \leq 4t$, and then the dual elements $v_i^*$ and $f_j^*$ generate $C^0(Y;\Z)$ and $C^2(Y;\Z)$.

We now attempt to find an integral cocycle with bounded coefficients representing a nonzero class in $H^1(Y;\Z)$.  We observe that $\ker(\delta^0)$ has rank $b_0(Y) = 1$, so $\img(\delta^0)$ has rank $m-1$ and then 
\[\rank(\ker(\delta^1)) = b_1(Y) + \rank(\img(\delta^0)) = 1+(m-1) = m. \]
The operator $\delta^1$ thus has rank $n-m$.  Representing it as a $k\times n$ matrix $A$ in the standard bases $\big(e_i^*\big)$ and $\big(f_j^*\big)$ of $C^1(Y;\Z)$ and $C^2(Y;\Z)$, we can renumber the $f_j$ so that the first $n-m$ rows of $A$ are linearly independent.  Similarly, we can renumber the vertices $v_i$ so that $\delta^0v_1^*, \dots, \delta^0v_{m-1}^*$ are linearly independent in $C^1(Y;\Z)$.  We now form an $(n-1)\times n$ matrix $B$ whose first $n-m$ rows are the first $n-m$ rows of $A$, and whose remaining $m-1$ rows are the coordinates of the vectors $\delta^0v_1^*,\dots,\delta^0v_{m-1}^*$.

By construction, a nonzero element $x \in \Z^n$ belongs to $\ker(B)$ if and only if it is in $\ker(A)$ and orthogonal to each of the $\delta^0v_i^*$; the corresponding cocycle $\phi$ is not in $\img(\delta^0)$, so it represents a nontrivial cohomology class.  The entries of $B$ are bounded as follows: every entry in the first $n-m$ rows is an entry of $A$, so it has the form $a_{ij} = (\delta^1e_j^*)(f_i) = e_j^*(\partial f_i)$, which has absolute value at most 3.  Each entry in the last $m-1$ rows has the form $(\delta^0v_i^*)(e_j) = v_i^*(\partial e_j)$, which is $\pm 1$ if exactly one endpoint of $e_j$ is $v_i$ and $0$ otherwise.  Thus $B$ is an $(n-1)\times n$ matrix whose entries are all integers between $-3$ and $3$, so Siegel's lemma \cite[p.\ 213]{siegel} says that $Bx=0$ has a nonzero integer solution $x = \langle x_1,\dots,x_n\rangle$ satisfying
\[ |x_i| \leq (3n)^{(n-1)/(n-(n-1))} = (3n)^{n-1} \]
for all $i$.  In particular, this gives a cocycle $\phi \in C^1(Y;\Z)$ with $[\phi] \neq 0$ in $H^1(Y;\Z)$ and $||\phi||_{L^1} \leq (3n)^{n-1}\cdot n = \frac{1}{3}(3n)^{n} \leq \frac{1}{3}(18t)^{6t}$.

Now suppose that $[\phi] = d[\phi']$, where $[\phi']$ generates $H^1(Y;\Z)$; if $d=1$ then we are done, so we may assume that $d \geq 2$.  Then we have integers $c_1,\dots,c_m$ such that
\[ \phi = d\phi' + \sum_{i=1}^m c_i(\delta_0v_i^*). \]
Letting $n_i$ be the closest integer to $\frac{c_i}{d}$ for each $i=1,\dots,m$, we define a cocycle
\[ \phi_0 = \phi' + \sum_{i=1}^m n_i(\delta_0v_i^*) \in C^1(Y;\Z) \]
and observe that $[\phi_0] = [\phi']$ generates $H^1(Y;\Z)$, and also that
\[ \left|\left|\phi_0 - \frac{1}{d}\phi\right|\right|_{L^1} = \left|\left|\sum_{i=1}^m \left(n_i-\frac{c_i}{d}\right)(\delta_0v_i^*)\right|\right|_{L^1} \leq \frac{1}{2} \sum_{i=1}^m \left|\left|\delta_0v_i^*\right|\right|_{L^1}
\]
by the triangle inequality.  Now $\left|\big(\delta_0v_i^*\big)(e_j)\right| \leq 1$ as explained above, so $||\delta_0v_i^*||_{L^1} \leq n$ and hence the right hand side above is at most $\frac{1}{2}mn \leq 12t^2$.  We conclude by the triangle inequality that
\[ ||\phi_0||_{L^1} \leq \left|\left|\phi_0 - \frac{1}{d}\phi\right|\right|_{L^1} + \frac{1}{d}||\phi||_{L^1} \leq 12t^2 + \frac{\frac{1}{3}(18t)^{6t}}{2}, \]
which is clearly at most $\frac{1}{3}(18t)^{6t}$ for all $t \geq 1$, establishing the claim.
\end{proof}

With the cocycles provided by Proposition~\ref{prop:build-cocycle} in hand, we can now certify the cabledness or compositeness of a given knot.

\begin{theorem} \label{thm:cabled-np}
Given a diagram of a cabled knot $K$ with $n$ crossings, there is a certificate which can be used to verify in $\poly(n)$ time that $K$ is a nontrivial cable.
\end{theorem}

\begin{proof}
We construct a triangulation $\cT$ of the exterior $E_K$ with $t=O(n)$ tetrahedra, whose restriction to $\partial E_K$ consists of two triangles and a single vertex, using Proposition~\ref{prop:triangulate-exterior}.  We can realize a cabling annulus $A$ as a normal surface with $||v(A)||_{L^1} \leq 7t \cdot 2^{7t}$ by Proposition~\ref{prop:small-normal-annuli}.  We can use Proposition~\ref{prop:annulus-triangulation} to produce a triangulation of $E_K \ssm N(A)$ with at most $200t$ tetrahedra in $\poly(t)$ time.

The complement $E_K \ssm N(A)$ consists of two components, one of which is a solid torus and the other of which is the complement of the companion $K'$; we will call these $C$ and $E$ respectively.  We can use the triangulation of $E$ to find a presentation of $\pi_1(E)$ in $\poly(t) = \poly(n)$ time.  Our certificate that $K$ is cabled thus contains the following:
\begin{enumerate}
\item Normal coordinates $v(A)$ of a properly embedded, essential annulus $A \subset E_K$ with weight $||v(A)||_{L^1} \leq 7t \cdot 2^{7t}$;
\item A certificate that one component $C$ of $E_K \ssm N(A)$ is a solid torus, as provided by Ivanov \cite{ivanov};
\item For the other component $E$ of $E_K \ssm N(A)$, a simplicial 1-cocycle $\phi$ with integer coefficients satisfying $||\phi||_{L^1} < \frac{1}{3}(18\cdot 200t)^{6\cdot 200t}$;
\item A positive integer $\theta$, with $1 \leq \theta \leq n$;
\item A certificate that the Poincar\'e dual of $[\phi]$ has Thurston norm $\theta$, as provided by Lackenby \cite{lackenby-knottedness}.
\end{enumerate}
We verify the certificate as follows:
\begin{enumerate}
\item Verify that $\Delta_K(t) \neq 1$. \label{i:cable-alex-poly}
\item Use the presentation of $\pi_1(E)$ to compute the Alexander polynomial of $E$ via Fox calculus, and verify that it has strictly smaller degree than $\Delta_K(t)$. \label{i:cable-alex-2}
\item Verify that $A$ is an annulus and that the components of $\partial A$ are homologically nontrivial in $\partial E_K$. \label{i:cable-annulus-pi1}
\item Verify as in \cite{ivanov} the certificate that $C \cong S^1 \times D^2$.
\item Verify the certification of the Thurston norm of $[\phi]$, as in \cite{lackenby-knottedness}. \label{i:cable-companion-norm}
\end{enumerate}

To see that this correctly proves that $K$ is cabled, we note that if $K$ is a $(p,q)$-cable of $K'$, then we have the relation
\[ \Delta_K(t) = \Delta_{T_{p,q}}(t) \Delta_{K'}(t^q), \]
and so $\Delta_K(t) \neq 1$ and $\deg(\Delta_K(t)) > \deg(\Delta_{K'}(t))$.  Thus step~\eqref{i:cable-alex-poly} certifies that $K$ is knotted without incorrectly eliminating any cables.  Since $K$ is knotted, step~\eqref{i:cable-annulus-pi1} proves that $A$ is either essential or boundary-parallel.  In the latter case, we have $E \cong E_K$, so the Alexander polynomial computed in step~\eqref{i:cable-alex-2} will have the same degree as $\Delta_K(t)$.  In particular, the certificate verifies that $A$ is essential.

Since $A$ is an essential annulus, $K$ must be a torus knot, a cable, or composite.  If it is composite then both components of $E_K \ssm N(A)$ are nontrivial knot exteriors, so $C$ cannot be a solid torus.  If instead $K$ is a torus knot, then both components are solid tori, so the Thurston norm of $E$ vanishes.  Thus if the certificate exists, it proves that $K$ is a cable.

Now we need to see that a certificate exists whenever $K = C_{p,q}(K')$ is a nontrivial cable of a nontrivial knot $K'$.  The normal annulus $A$ exists by Proposition~\ref{prop:small-normal-annuli}, as does the certificate that $C \cong S^1 \times D^2$.  Using Proposition~\ref{prop:build-cocycle}, we know that there is a cocycle $\phi$ representing a generator of $H^1(E) \cong \Z$, where $||\phi||_{L^1}$ satisfies the given bounds since $E$ has at most $200t$ tetrahedra.  The Poincar\'e dual of $\phi$ then has Thurston norm $\theta = 2g(K')-1 > 0$, where $g(K')$ is the Seifert genus of the companion $K'$.  But we have
\[ 2g(K')-1 \leq 2g(K)-1 \leq n, \]
where $g(K') \leq g(K)$ by \cite{shibuya} and $2g(K)-1 \leq n$ by an easy application of Seifert's algorithm, and so $1 \leq \theta \leq n$.  Thus $\phi$ and $\theta$ exist as well.

Finally, it is mostly clear that each step of the verification requires only polynomial time in $n$, since $t = O(n)$.  We recall that in step~\eqref{i:cable-annulus-pi1}, we verify that $A$ is an annulus using \cite{agol-hass-thurston}, and that we determine that $\partial A$ consists of essential curves in $\partial E_K$ using their characterization as normal curves as in Section~\ref{sec:normal-surfaces}.  Moreover, step~\eqref{i:cable-companion-norm} takes $\poly(200t, \log(\theta), \log(||\phi||_{L^1}))$ time by \cite[Theorem~1.5]{lackenby-knottedness}; but $\log(\theta) = O(\log n)$ and $\log(||\phi||_{L^1}) = O(t\log t)$, and $t = O(n)$, so this is again $\poly(n)$ as desired.
\end{proof}

The construction of certificates proving that a knot $K$ is composite, and the algorithm which verifies them in polynomial time, is nearly identical.  The only complication is that composite knots can have Alexander polynomial 1, so we check that they are knotted by using Lackenby's knottedness certificates \cite{lackenby-knottedness} instead of their Alexander polynomials.

\begin{theorem} \label{thm:composite-np}
Given a diagram of a composite knot $K$ with $n$ crossings, there is a certificate which can be used to verify in $\poly(n)$ time that $K$ is a nontrivial connected sum.
\end{theorem}

\begin{proof}
We produce a triangulation $\cT$ of the exterior $E_K$ which restricts to a one-vertex triangulation on $\partial E_K$ by Proposition~\ref{prop:triangulate-exterior}, and given a normal essential annulus $A$ which splits $E_K$ into a pair of nontrivial knot complements $E_1$ and $E_2$ we can triangulate $E_K \ssm N(A)$ with $O(n)$ tetrahedra in $\poly(n)$ time by Proposition~\ref{prop:annulus-triangulation}.  We can arrange that $||v(A)||_{L^1} \leq 7t \cdot 2^{7t}$, where $\cT$ has $t = O(n)$ tetrahedra, by Proposition~\ref{prop:small-normal-annuli}.  Using the triangulations of $E_1$ and $E_2$, we can similarly produce presentations of $\pi_1(E_1)$ and $\pi_1(E_2)$ in $\poly(n)$ time.  All of this is accomplished exactly as in the proof of Theorem~\ref{thm:cabled-np}.

The compositeness certificate for the given diagram consists of:
\begin{enumerate}
\item A certificate that it represents a nontrivial knot, as specifed by Lackenby \cite{lackenby-knottedness};
\item Normal coordinates for an essential annulus $A$, with $||v(A)||_{L^1} \leq 7t \cdot 2^{7t}$;
\item For both $i=1$ and $i=2$:
\begin{enumerate}
\item A simplicial 1-cocycle $\phi_i$ on $E_i$ with integral coefficients, satisfying $||\phi_i||_{L^1} < \frac{1}{3}(3600t)^{1200t}$;
\item A positive integer $\theta_i < n$;
\item A certificate that the Poincar\'e dual of $[\phi_i]$ has Thurston norm $[\theta_i]$, as provided by \cite{lackenby-knottedness}.
\end{enumerate}
\end{enumerate}
We use \cite{lackenby-knottedness} to verify that $K$ is knotted, and \cite{agol-hass-thurston} to verify that $v(A)$ indeed represents an annulus.  We check that the components of $\partial A$ are homologically essential in $\partial E_K$, which then proves that $A$ is either essential or boundary-parallel.  We then verify that the classes dual to $[\phi_i]$ on each $E_i$ have the given Thurston norms.  All of this can be done in $\poly(n)$ time, as in Theorems~\ref{thm:torus-knot-np} and \ref{thm:cabled-np}.

To see that this correctly verifies that $K$ is composite, we note that neither component of $E_K \ssm N(A)$ can be a solid torus, since then its Thurston norm would vanish.  This implies that $A$ cannot be boundary-parallel, so it must be essential, and then since neither component is a solid torus we conclude that $K$ is composite.  Moreover, if $K$ is composite then a certificate exists: the normal annulus exists by Proposition~\ref{prop:small-normal-annuli}, and if $E_i$ is the exterior of a knot $K_i$ then we have $g(K_1) + g(K_2) = g(K)$, so that 
\[ 1 \leq 2g(K_i)-1 \leq 2g(K)-1 < n \]
and hence the desired bounds on $||\phi_i||_{L^1}$ and $\theta_i$ follow as in Theorem~\ref{thm:cabled-np}.
\end{proof}

\bibliographystyle{halpha}
\bibliography{References}

\end{document}